\documentclass{article}

\usepackage{amsthm}
\usepackage{amsfonts}
\usepackage{amsmath}
\usepackage{graphicx}
\usepackage{enumerate}
\usepackage{amssymb}
\usepackage{hyperref}
\usepackage{latexsym}
\usepackage{graphicx}
\hypersetup{
    colorlinks=true,
    linkcolor=blue,
    filecolor=magenta,      
    urlcolor=cyan,
}

\newtheorem{theorem}{Theorem}[section]
\newtheorem{exmp}{Example}[section]

\newtheorem{lemma}[theorem]{Lemma}

\newtheorem{definition}[theorem]{Definition}

\newcommand{\Z}{\mathbb{Z}}

\newcommand{\Q}{\mathbb{Q}}

\title{The Farey Sequence, Stern-Brocot Tree and Euclid’s Theorem}
\author{Charles Alba \and Nathan Hill}
\date{\today}

\begin{document}

\maketitle

\section{Introduction}

Farey's sequence, invented by John Farey, is a procedure used to generate proper fractions from 0 to 1 (from citation \cite{citation01}). Farey sequence is commonly used in rational approximations of irrational numbers, ford circles and in Riemann hypothesis (from citation \cite{citation02}). In this paper, we aim to use properties of the Farey's sequence to prove the popular gcd theorem. We will define a sequence of subsets using a method called the Mediants Property, defined by theorem \ref{def:Mediant Property}. Most of our definitions and related works will be based on the works of Ainsworth, Dawson, Pianta, Warwick as well as Knott, found under citations \cite{citation03}, \cite{citation04}. 

\begin{definition}[Farey Sequence]
\label{def:Farey's sequence}
The {\it Farey's sequence} is a sequence of subsets $F_n \in \Q$ defined by
 $F_n=\{ \frac{a}{b}:\frac{a}{b} \in \Q , 0 \le a \le b \le n \}$
\end{definition}

\begin{exmp}
\label{exmp: 1.1}
\[F_1 = \left\lbrace\frac{0}{1},\frac{1}{2},\frac{1}{1}\right\rbrace \]
\[F_2 =  \left\lbrace\frac{0}{1},\frac{1}{3},\frac{1}{2},\frac{2}{3},\frac{1}{1}\right\rbrace \]
\[F_3 =  \left\lbrace\frac{0}{1},\frac{1}{4},\frac{1}{3},\frac{1}{2},\frac{2}{3},\frac{3}{4},\frac{1}{1}\right\rbrace \]
\end{exmp}

\begin{definition}[order]
 \label{def:Order} 
In the Farey's series, the {\it order} of a series is the size of maximum denominator of a series. 
\end{definition}

\begin{exmp}
Using Example \ref{exmp: 1.1}, \\
$F_1$ has order 2, \\$F_2$ has order 3, \\$F_3$ has order 4
\end{exmp}

Now that we have defined the Farey's sequence, we would like to introduce the Mediant, which is a key property of Farey's sequence that would be important towards the Stern-Brocot Tree found in section \ref{sec:Stern-Brocot Tree}. 

\begin{definition}[Mediant]
\label{def:Mediant}
The {\it Mediant} of two fractions can be expressed as: 
\[ \frac{a}{b} \bigoplus  \frac{c}{d} =  \frac{a+c}{b+d} \] 
\end{definition}

The mediant has an important property that we will later prove and use to prove the gcd theorem in section \ref{sec:Stern-Brocot Tree}

\begin{theorem} [Mediant Property]
\label{def:Mediant Property}
The {\it Mediant Property} states that if $\frac{a}{b} \textless \frac{c}{d}$ then their mediant $\frac{a+c}{b+d}$ lies between them, $\frac{a}{b} \textless \frac{a+c}{b+d} \textless \frac{c}{d}$
\end{theorem}

\begin{definition}[Euclidean Algorithm]
\label{def:Euclidean Algorithm}
The {\it Euclidean Algorithm} is a process for determining the gcd of two $\Z$, $a$ and $b$.\\
Define $b=q_1a+a_1$\\
 $a=q_2a_1+a_2$\\
 $a_1=q_3a_2+a_3$\\
...\\
 $a_n=q_{n+2}a_{n+1}$\\
Where $a_{n+1}=gcd(a,b)$
\end{definition}

\begin{theorem}[GCD theorem]
\label{def: gcd theorem} 
The {\it gcd theorem}  states that for coprime numbers $m,n \in \Z$, there exists $x,y \in \Z$, such that $xm+yn=1$, where $1=gcd(m,n)$.
\end{theorem}

Now thar we have defined the definitions and theorems of the Farey Sequence and Euclid's theorem, we will establish some important tools and lemmas of the Stern-Brocot Tree in section \ref{sec:Stern-Brocot Tree} and show that we can proof the gcd theorem using the Farey's sequence in section \ref{sec: proof}.

\section{More about the Farey's Sequence}
\label{sec:Farey's Sequence}

We shall start off by proving the Mediant property, which is the key property of the Farey's sequence. 
\begin{proof}
Since $\frac{a}{b} \textless \frac{a+c}{b+d} \textless \frac{c}{d}$. Therefore $\frac{a}{b} \textless \frac{a+c}{b+d} $, so:
\[\frac{a}{b} \textless \frac{a+c}{b+d} \]
\[\frac{a+c}{b+d} - \frac{a}{b} \textgreater 0\]
\[\frac{bc-ad}{b(b+d)}\textgreater 0\]
\\ And likewise, with $\frac{a+c}{b+d} \textless \frac{c}{d}$
\[ \frac{a+c}{b+d} \textless \frac{c}{d} \]
\[\frac{c}{d} - \frac{a+c}{b+d} \textgreater 0\]
\[\frac{bc-ad}{d(b+d)}\textgreater 0\]
\end{proof}

Proof adopted from Ainsworth, Dawson, Pianta, Warwick, found in citation \cite{citation03}\\

The Farey sequence can be constructed in the following manor:\\
The first row,$F_1$, contains $\frac{0}{1}$ and $\frac{1}{1}$. 
From here, assume the nth row has been completed, the nth+1 row is made by copying the nth row, then for each pair of consecutive fractions $\frac{a}{b}$ and $\frac{c}{d}$ that are in the nth row that sitisfies $b+d\leq n+1$,
 insert inbetween $\frac{a}{b}$ and $\frac{c}{d}$ the new fraction $\frac{a+c}{b+d'}$.\\
So, for row 2, we look at the only pair of consecutive fractions $\frac{0}{1}$ and $\frac{1}{1}$. $1+1 \leq 1+1$, so the pair satifies the condition, we now insert the new fraction $\frac{0+1}{1+1}=\frac{1}{2}$ in the new row, giving $\frac{0}{1} \frac{1}{2} \frac{1}{1}$.\\
Here is a filled out table up to $F_6$\\
\begin{center}
\includegraphics[width=3in]{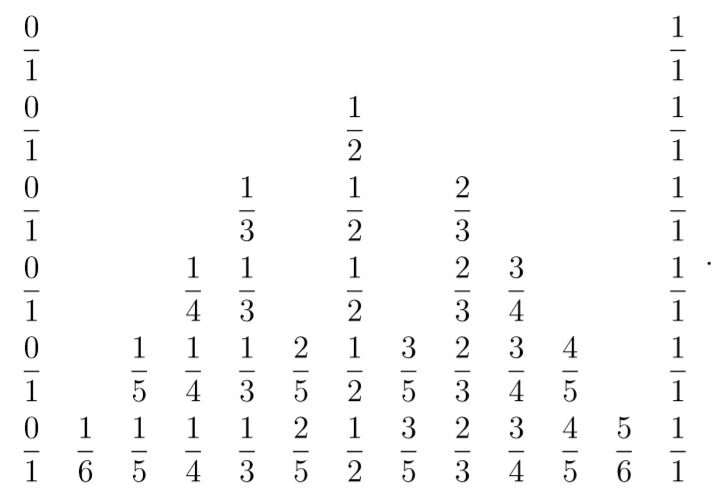}
\end{center}

\section{Stern-Brocot Tree}
\label{sec:Stern-Brocot Tree}

Previously, in section \ref{sec:Farey's Sequence}, we defined the Fareys sequence and identified properties of the Fareys sequence. In this section, we hope to further use these properties of theorem \ref {def:Mediant Property}, the mediant property, in the Stern-Brocot Tree, which is a further expansion from Farey's sequence.\\

\begin{definition}[Stern-Brocot Tree]
\label{def:Stern-Brocot Tree}
Let $S_0 = \{ 0,1\}$. Given $S_n$, let $S_n'$ denote the set of mediants of adjacent elements of $S_n$. Then let $S_{n+1} = S_n \cup S_n'$.\\
The {\it Stern-Brocot Tree} is an infinite graph whose vertex set is $V = \bigcup_{n=0}^\infty S_n$,
 and where there is an edge between vertices $v$ and $w$ if and only if $v \in S_n'$, $w \in S_{n+1}$, and $w$ is obtained as a mediant of $v$ and some other element of $S_n$.
\end{definition}

\begin{definition}[Ancestor]
\label{def:Ancestor} 
The {\it Ancestor} of a number in the Stern-Brocot Tree,defined in definition \ref{def:Stern-Brocot Tree}, is any number that appears above it in the same branch. 
\end{definition}

\begin{definition}[Child]
\label{def:Child}
\label{def:Child}
The {\it Child} of a number in the Stern-Brocot Tree is any number that appears below it. Each number has a Right Child ($R$) and a Left Child ($L$) 
\end{definition}

\begin{lemma}
The Left Child of any number is the mediant of that number and its Ancestor to the left. 
\end{lemma}

\begin{exmp}
\label{exmp: 3.1}
Here is an example of the Stern-Brocot Tree.
\begin{center}
\includegraphics[width=5in]{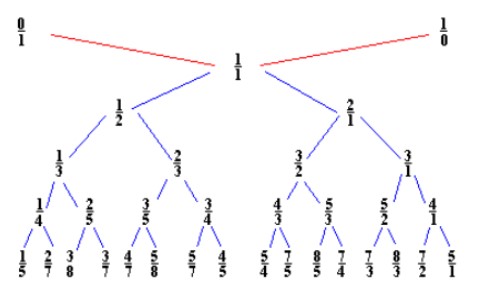}
\end{center}
\end{exmp}

\section{Using Farey's Sequence to Prove GCD theorem}
\label{sec: proof}

\begin{definition}[Reduced Form] 
\label{definition: reduced form}
A fraction is said to be in its {\it Reduced Form} if the fraction $\frac{m}{n}$, where $m,n \in \Z$ is expressed in the lowest terms. Therefore, $m$ and $n$ have to be coprime.
\end{definition}

\begin{theorem}
\label{theorem: 1}
For every $\frac{m}{n}$ appearing in the Stern-Brocot Tree in the reduced form, there exist $x,y$ such that $mx+ny=1$, where $x,y \in \Z$
\end{theorem}
\begin{proof}
Recalling from definition \ref{def:Stern-Brocot Tree}, we rename $S_n$ as $S_k$ We claim by induction that the theorem holds. 
\\For base case, $S_1=\{ \frac{0}{1},\frac{1}{2},\frac{1}{1} \}$, notice that $(1)(1)-(2)(0)=(2)(1)-(1)(1)=(2)(1)-(1)(1)=1$, so the result is true for $k=1$. 
\\Now to test $S_{k+1}$. Let $a=\frac{m_1}{n_1}$ and $b=\frac{m_2}{n_2}$, where $m_1,m_2,n_1,n_2 \in \Z^{nonneg}$. $S_{k+1} = \{ \dots \frac{m_1}{n_1}, \frac{m_1+m_2}{n_1+n_2} , \frac{m_2}{n_2}, \dots \}$. We know that $0 \leq\frac{m_1}{n_1} \leq \frac{m_1+m_2}{n_1+n_2} \leq  \frac{m_2}{n_2} \leq 1$. 
Lets zoom in on  $0 \leq\frac{m_1}{n_1} \leq \frac{m_2}{n_2} \leq 1$, which gives us $0 \leq  n_1m_2-m_1n_2 \leq 1$. $(n_1m_2-m_1n_2)$ can hence only be $0$ or $1$ since $m_1,m_2,n_1,n_2 \in \Z^{nonneg}$. $n_1,n_2 \neq 0$ and $n_1m_2\ne m_1n_2 $, so therefore $n_1m_2-m_1n_2=1$. 
\\Notice, $S_{k+1} = \{ \dots \frac{m_1}{n_1}, \frac{m_1+m_2}{n_1+n_2} , \frac{m_2}{n_2}, \dots \}$, we have $n_1(m_1+n_2)-m_1(n_1+n_2)=n_1m_2-m_1n_2=1$, likewise, $m_2(n_1+n_2)-n_2(m_1+n_2)=n_1m_2-m_1n_2=1$. So this holds true for $k+1$.
\end{proof}

\begin{definition}[consecutive]
\label{def:consecutive} 
We say that two fractions $\frac{m_1}{n_1}$ and $\frac{m_2}{n_2}$ are {\it consecutive} if one fraction is a direct child (defined in definition \ref{def:Child}) of the other fraction at any stage of the construction of the Stern-Brocot Tree. 
\end{definition}

\begin{theorem}
\label{theorem: 2}
Every rational number $\frac{m}{n}$ in the Stern-Brocot Tree appears in its reduced form and appears exactly once.
\end{theorem}
\begin{proof}
Let $\frac{m_1}{n_1}$ and $\frac{m_2}{n_2}$ be consecutive fractions. Recalling Theorem \ref{theorem: 1} then,
\[ m_2n_1-m_1n_2=1 \label{eq:1} \tag{1}\]
Therefore, based on the definition of mediant (found in theorem \ref{def:Mediant Property}), the mediant of the 2 fractions should satisfy:
\[ n_1(m_1+m_2) - (n_1+n_2)m_1 = 1 \label{eq:2} \tag{2}\ \]
\[ m_2(n_1+n_2)-(m_1+m_2)n_2 = 1\label{eq:3} \tag{3} \]

Note that both (2) and (3) are equivalent to (1) if we compute it as follows:
\[ For\ (\ref{eq:2}): n_1(m_1+m_2) - (n_1+n_2)m_1 =  n_1m_1+n_1m_2 - n_1m_1-m_1n_2 =m_2n_1-m_1n_2=1\]
\[ For\ (\ref{eq:3}): m_2(n_1+n_2)-(m_1+m_2)n_2 =  n_1m_2+n_2m_2 - n_2m_1-m_2n_2 =m_2n_1-m_1n_2=1\]
\\ Also if $\frac{m_1}{n_1} \textless \frac{m_2}{n_2}$ then 
\[ \frac{m_1}{n_1} \textless \frac{(m_1+m_2)}{(n_1+n_2)} \textless\frac{m_2}{n_2} \]
this is based of the Stern-Brocot Tree using Mediants. 
\\Finally, let $\frac{a}{b}$ be a fraction in lowest terms where $a,b \textgreater 0$, to show that this fraction will appear somewhere in the tree
\[ \frac{m_1}{n_1} \textless \frac{a}{b} \textless \frac{m_2}{n_2} \]
Expanding the equation, we get
\[ m_1n_2b \textless an_1n_2 \textless m_2n_1b \]
Seperating the inequalities, we can get
\[ m_1n_2b - an_1m_2 \textless 0 \ and\ an_1n_2 - m_2bn_1 \textless 0 \]
This can be reduced to:
\[n_1a-m_1n \geq 1 \ and\ m_2b-n_2a \geq 1\]

from where
\[ (m_2+n_2)(m_1a-m_1b)+(m_1+n_1)(m_2b-n_2a) \geq m_1+n_1+m_2+n_2 \]

Where expanding it we get:
\[ m_2n_1a - m_1n_2a + m_2n_1b - m_1n_2b \geq m_1+n_1+m_2+n_2  \]

Using (\ref{eq:1}), 
\[ a+b \geq m_1+n_1+m_2+n_2 \]

This shows that there would be at most (a+b) steps of mediants before $\frac{a}{b}$ appears in the Stern-Brocot Tree. 
\end{proof}
Proof adopted from Bogomolny, found in citation \cite{citation06}.

\begin{theorem}
\label{theorem: 3}
If $m$ and $n$ are coprime, there exists $x, y$ such that $mx+ny=1$, where $x,y,m,n \in \Z$ 
\end{theorem}

\begin{proof}
Consider the fraction $\frac{m}{n}$. By theorem $\ref{theorem: 2}$, it appears that since $m$ and $n$ are co-primes, $\frac{m}{n}$ can appear in its already reduced form (refer to def \ref{definition: reduced form}) in the Stern-Brocot Tree. Therefore, by theorem $\ref{theorem: 1}$, we conclude that for $\frac{m}{n}$ in the Stern-Brocot Tree, there exist $x,y$ such that $mx+ny=1$, where $x,y \in \Z$. This obeys the GCD theorem as defined by Definition \ref{def: gcd theorem}. 
\end{proof}

\end{document}